\documentclass{article}
\usepackage{amsmath}
\usepackage{amsthm}
\usepackage{amsfonts}
\usepackage{latexsym}
\newcounter{alphthm}
\newtheorem{thm}{Theorem}

\newtheorem{cor}{Corollary}
\newtheorem{lem}{Lemma}

\newtheorem{rem}{Remark}

\newcommand{\be}{\begin{equation}}
\newcommand{\ee}{\end{equation}}
\newcommand{\ben}{\begin{enumerate}}
\newcommand{\een}{\end{enumerate}}
\newcommand{\pa}{{\partial}}

\newcommand{\g}{{\bf g}}
\newcommand{\pxi}{{\pa \over \pa x^i}}

\def\beq{\begin{equation}}
\def\eeq{\end{equation}}

\title{\Large \textbf{On Ricci-Tensor of Randers Metrics}\footnote{Accepted in Journal of Geometry and Physics, 2010.}}
\author{A. Tayebi and E. Peyghan}

\begin{document}

\maketitle
 \begin{abstract}
In this paper,  we study  Randers metrics and  find a condition on Ricci tensor of these metrics to be Berwaldian. This generalize  Shen's Theorem which says: every R-flat complete Randers metric is locally Minkowskian. Then we find a necessary and sufficient condition on Ricci tensor under which a Randers  metric of scalar flag curvature is of zero flag curvature.
\end{abstract}

\section{Introduction.}
For a Finsler metric $F=F(x,y)$, its geodesics are characterized by the system
of differential equations $ \ddot c^i+2G^i(\dot c)=0$, where the local functions $G^i=G^i(x, y)$ are called the spray coefficients. A Finsler metric $F$ is  called a Berwald metric if $G^i = {1\over 2} \Gamma^i_{jk}(x)y^jy^k$ are quadratic in $y\in T_xM$  for any $x\in M$. It is  proved that on  a Berwald space,  the parallel translation along any geodesic preserves the Minkowski  functionals \cite{Ic}. Thus Berwald spaces can be viewed as  Finsler spaces modeled on a single Minkowski space.

In order to find explicit examples of  Berwald metrics, we consider Randers metrics.  By definition a Randers metric is a scalar function on $TM$ defined by $F=\alpha +\beta$  where   $\alpha=\sqrt{a_{ij}(x)y^iy^j}$ is a Riemannian metric and $\beta =b_i(x)y^i$ is a 1-form on  $M$.  The Randers metrics were introduced by G. Randers in the context of general relativity \cite{Ra}. For a Randers metric $F=\alpha +\beta$, it is proved that  if $\beta$ is parallel with respect to $\alpha$, then $F$ is a Berwald metric \cite{HaIc1}.

In \cite{ShLandsverg}, Shen prove that  every regular $(\alpha, \beta)$-metric with vanishing Landsberg curvature is a Berwald metric.  Therefore it is interesting to find another Finslerian curvature  with the same property. For this work, we  study  the Riemannian curvature and  its averaging quantity, i.e, the Ricci curvature of Randers metrics.

A Randers metric and its Ricci curvature, are related by their histories in physics. The well-known Ricci curvature was introduced by  Ricci which used to formulate the Einstein's of gravitation.  Nine years later, Ricci's work was used to formulate the Einstein's of gravitation.  Recently  Robles investigated the  Ricci curvature of  Randers metrics and obtained the necessary and sufficient conditions for a Randers metric to be Einstein \cite{BR1}\cite{Rob}. This is one of our motivations  to investigation on the Ricci tensor of Randers metrics.

In this  paper,  we study the Randers metric and its Ricci tensor. We find a condition on Ricci tensor of Randers metric $F=\alpha +\beta$, such that the mean Landsberg tensor  of $F$ satisfying in a ODE (see Lemma \ref{PropRanders}).  By this ODE, we get a condition for complete Randers metric to be Berwald metric (Theorem \ref{thmRanders}). For this reason, let us define
\[
\mathfrak{R}_{ij}:=R_{ij}-\frac{1}{n+1}\Big(I^kR_{ijk}+I^kR_{kji}+F^{-1}\ell_iI^kR_{k0j}+I_iR_{0j}\Big),
\]
where $R_{ijkl}$ is the Riemannian curvature of the  Cartan connection, $R_{ij}:=R^{\ r}_{i\ jr}$ is the Ricci tensor, $R_{i0k}:=R_{ijk}y^j$, $R^{\ r}_{i\ jr}:= g^{rm}R_{imjr}$,  $R_{0k}:=R_{jk}y^j$, $\ell_i:=F_{y^i}$, $I_i:=g^{kl}C_{kli}$ is the mean Cartan tensor, $C_{kli}$ is the Cartan tensor, $g^{ij}$ is the inverse of the fundamental tensor $g_{ij}$ and $I^k=g^{ki}I_i$ \cite{Shim}. By definition, $\mathfrak{R}_{ij}$ is not a symmetric tensor. But in a Riemannian space,  $\mathfrak{R}_{ij}$ is equal to  $R_{ij}$ which is a symmetric tensor. It is  interesting to consider non-Riemannian Finsler spaces with vanishing tensor  $\mathfrak{R}_{ij}$.

According to  \cite{ShLandsverg}, every Randers metric with vanishing Landsberg curvature is a Berwald metric. Is there any other interesting Finslerian quantity which has the same property for Randers metrics? We will show that the  $\mathfrak{R}$-tensor is another candidate. A Finsler metric $F$ is called $\mathfrak{R}$-flat if $\mathfrak{R}_{ij}=0$. We prove the following.
\begin{thm}\label{thmRanders2}
Every complete $\mathfrak{R}$-flat Randers metric $F=\alpha +\beta$ on a manifold $M$  is Berwaldian.
\end{thm}
A Finsler metric $F$ is called  R-flat, if $R_{ijkl}=0$. Then every  R-flat metric is $\mathfrak{R}$-flat. In \cite{Sh1}, it is proved that every   R-flat Berwald metric is locally Minkowskian.  Therefore by  Theorem \ref{thmRanders2}, we get the following corollary.
\begin{cor}\label{Shen}
Every complete  R-flat Randers metric $F=\alpha +\beta$ on a manifold $M$  is locally Minkowskian.
\end{cor}
The Corollary \ref{Shen}  was proved by Prof. Z. Shen as Theorem 1.2 in \cite{ShKS}. Then  Theorem \ref{thmRanders2}  can be regarded as a generalization of  Theorem 1.2 in \cite{ShKS}.

\bigskip

As we mentioned, $\mathfrak{R}$-tensor is not symmetric. Therefore, it is a natural problem, that we find some conditions under which the $\mathfrak{R}$-tensor became symmetric. In section 3, we study  Randers metric of scalar flag curvature ${\bf K}$ with symmetric $\mathfrak{R}$-tensor. To our surprise that for a non-Riemannian  Randers metric of scalar flag curvature ${\bf K}$,  the $\mathfrak{R}$-tensor is symmetric if and only if ${\bf K}=0$.
\begin{thm}\label{RS2}
Let $F=\alpha+\beta$ be a non-Riemannian Randers metric of scalar flag curvature ${\bf K}$. Then  ${\bf K}=0$ if and only if\ \ $\mathfrak{R}$-tensor  is symmetric.
\end{thm}

There are many connections in Finsler geometry \cite{TAE}. Throughout this paper, we use  the Cartan connection on Finsler manifolds. The $h$- and $v$- covariant derivatives of a Finsler tensor field are denoted by `` $|$ " and ``, " respectively.


\section{Preliminaries}\label{sectionP}
Let $M$ be a n-dimensional $ C^\infty$ manifold. Denote by $T_x M $ the tangent space at $x \in M$, by $TM=\cup _{x \in M} T_x M $ the tangent bundle and by $TM_{0}:= TM \setminus\{0\}$ the slit tangent bundle  of $M$.

Let  $x\in M$ and $F_x:=F|_{T_xM}$.  To measure the non-Euclidean feature of $F_x$, define ${\bf C}_y:T_xM\times T_xM\times T_xM\rightarrow \mathbb{R}$ by ${\bf C}_y(u,v,w):=C_{ijk}(y)u^iv^jw^k$ where
\[
C_{ijk}(y):=\frac{1}{4}{{\partial^3 F^2} \over {\partial y^i \partial y^j
\partial y^k}}(y).
\]
The family ${\bf C}:=\{{\bf C}_y\}_{y\in TM_0}$  is called the Cartan torsion. It is well known that {\bf{C}=0} if and only if $F$ is Riemannian \cite{Sh1}. Define  mean Cartan torsion ${\bf I}_y$ by ${\bf I}_y(u):=I_i(y)u^i$, where $I_i:=g^{jk}C_{ijk}$ and $g^{jk}:=(g_{jk})^{-1}$. By Deicke's Theorem, {\bf{I}=0} if and only if $F$ is Riemannian \cite{Sh1}.

\smallskip

For  $y \in T_xM_0$, define the  Matsumoto torsion ${\bf M}_y:T_xM\otimes T_xM \otimes T_xM \rightarrow \mathbb{R}$ by ${\bf M}_y(u,v,w):=M_{ijk}(y)u^iv^jw^k$ where
\[
M_{ijk}:=C_{ijk} - {1\over n+1}  \{ I_i h_{jk} + I_j h_{ik} + I_k h_{ij} \}.\label{Matsumoto}
\]
A Finsler metric $F$ is said to be C-reducible if ${\bf M}_y=0$. This quantity is introduced by  Matsumoto \cite{Ma}. Matsumoto proves that every Randers metric satisfies that ${\bf M}_y=0$ \cite{M3}.

\smallskip

The horizontal covariant derivatives of ${\bf C}$ and $\bf{I}$ along geodesics give rise to  the  Landsberg curvature  ${\bf L}_y:T_xM\times T_xM\times T_xM\rightarrow \mathbb{R}$ and mean Landsberg curvature  ${\bf J}_y:T_xM\rightarrow \mathbb{R}$ defined by ${\bf L}_y(u,v,w):=L_{ijk}(y)u^iv^jw^k$ and ${\bf J}_y(u): = J_i (y)u^i$  where
\[
L_{ijk}:=C_{ijk|s}y^s\ \  \textrm{and}  \ \ J_i:=I_{i|s}y^s.
\]
The families ${\bf L}:=\{{\bf L}_y\}_{y\in TM_{0}}$ and ${\bf J}:=\{{\bf J}_y\}_{y\in TM_{0}}$ are called the Landsberg curvature and mean Landsberg curvature. A Finsler metric is called  Landsberg metric and  weakly Landsberg metric if {\bf{L}=0} and ${\bf J}=0$, respectively \cite{Sh1}.

\bigskip

Given a Finsler manifold $(M,F)$, then a global vector field ${\bf G}$ is
induced by $F$ on $TM_0$, which in a standard coordinate $(x^i,y^i)$
for $TM_0$ is given by
\[
{\bf G}=y^i {{\partial} \over {\partial x^i}}-2G^i(x,y){{\partial} \over
{\partial y^i}},
\]
where $G^i(x,y):={1\over 4}g^{il}(x, y)\{[F^2]_{x^ky^l}y^k-[F]^2_{x^l}\}$ are called the spray coefficients of ${\bf G}$.  The vector field ${\bf G}$ is called the  associated spray to $(M,F)$. The projection  of an integral curve of ${\bf G}$  to $M$ is a geodesic of  $(M, F)$. Using the spray coefficients of ${\bf G}$, one can define
\[
B^i_{\ jkl}(y):={{\partial^3 G^i} \over {\partial y^j \partial y^k
\partial y^l}}(y).
\]
For a  vector $y \in T_xM_0$, define the quantity ${\bf B}_y:T_xM\otimes T_xM \otimes T_xM \rightarrow T_xM$ by
$
{\bf B}_y(u, v, w):=B^i_{\ jkl}(y)u^jv^kw^l{{\partial } \over {\partial
x^i}}|_x,
$
The quantity ${\bf B}$ is called the Berwald curvature. A Finsler metric is called a Berwald metric  if {\bf{B}=0} \cite{Sh1}. Every Berwald metric is a Landsberg metric \cite{Sh1}.

\bigskip

The Riemann curvature ${\bf R}_y= R^i_{\ k}  dx^k \otimes \pxi|_x :
T_xM \to T_xM$ is a family of linear maps on tangent spaces, defined
by
\[
R^i_{\ k} = 2 {\pa G^i\over \pa x^k}-y^j{\pa^2 G^i\over \pa
x^j\pa y^k} +2G^j {\pa^2 G^i \over \pa y^j \pa y^k} - {\pa G^i \over
\pa y^j} {\pa G^j \over \pa y^k}.  \label{Riemann}
\]
For a flag $P={\rm span}\{y, u\} \subset T_xM$ with flagpole $y$, the  flag curvature ${\bf K}={\bf K}(P, y)$ is defined by
\[
{\bf K}(P, y):= {\g_y (u, {\bf R}_y(u)) \over \g_y(y, y) \g_y(u,u)
-\g_y(y, u)^2 },
\]
where $\g_y = g_{ij}(x,y)dx^i\otimes dx^j$.  We say that a Finsler metric $F$ is   of scalar curvature if for any $y\in T_xM$, the flag curvature ${\bf K}= {\bf K}(x, y)$ is a scalar function on  $TM_0$. If ${\bf K}=constant$, then $F$ is said to be of  constant flag curvature.

\bigskip

Let us consider the pull-back tangent bundle $\pi^*TM$ over $TM_0$ defined by $\pi^*TM=\left\{(u, v)\in TM_0 \times TM_0 | \pi(u)\\=\pi(v)\right\}$. Take a local coordinate system $(x^i)$ in $M$, the local natural frame $\{{{\partial} \over {\partial x^i}}\}$  of  $T_xM$ determines a local natural frame $\partial_i|_v$ for $\pi^*_vTM$ the fibers of  $\pi^*TM$, where  ${\partial _i |_v=(v,{{\partial} \over {\partial x^i}}| _x )}$, and $v=y^i{{\partial}\over {\partial x^i}}|_x\in TM_0$. The fiber $\pi^*_vTM$ is isomorphic to  $T_{\pi(v)}M$ where $\pi(v)=x$. There is a canonical section $\ell$  of $\pi^*TM$ defined by $\ell_v=(v,v)/F(v)$.

Now let $\nabla$ be the Cartan connection on  $\pi^*TM$ and $\{e_i\}^n _{i=1}$ be a  local orthonormal frame field for $\pi ^* TM$ such   that $e_n=\ell$. Let $\{\omega^i\}^n_{i=1}$ be its dual co-frame field.  Put   $\nabla e_i = \omega ^{\ j} _i \otimes e_j$ and $ \Omega e_i=2\Omega ^{\ j} _i \otimes e_j$,   where $\{\Omega ^{\ j} _{i}\}$ and $\{\omega ^{\ j} _{i}\}$ are called respectively,  the  curvature forms and  connection forms of $\nabla$ with respect to   $\{e_{i}\}$. Put $\omega^{n+i}:=\omega^{\ i}_n +d(log F)\delta^i _n$.  Then $\{\omega ^i, \omega^{n+i} \}^n_{i=1}$ is a  local basis for $T^*( TM_0)$. Since $\{\Omega ^{\ j} _{i}\}$ are 2-forms on $TM_0$,  they can be expanded as
\[
\Omega^{\ j}_ i={1 \over 2}R^{\ j} _{i \ kl} \omega ^k \wedge
\omega^l +P^{\ j} _{i \ kl} \omega ^k \wedge \omega^{n+l}+{1 \over
2}Q^{\ j} _{i \ kl} \omega ^{n+k} \wedge \omega^{n+l}.
\]
Let $\{\bar e_i, \dot e_i\}^n _{i=1}$ be the local basis for $T(TM_0)$, which is dual to $\{\omega ^i, \omega^{n+i} \}^n  _{i=1}$. The objects $R$, $P$ and  $Q$ are called, respectively,  the hh-, hv- and vv-curvature  tensors of the Cartan connection  with the components $R(\bar e_k,\bar e_l)e_i =R^{\ j}_{i \ kl}e_j$, $ P(\bar e_k,\dot e_l)e_i=P^{\ j}_{i \ kl} e_j$ and $Q(\dot e_k,\dot e_l)e_i=Q^{\ j} _{i \ kl} e_j$ (\cite{TAE}).

\section{Proof of Theorem \ref{thmRanders2}.}\setcounter{equation}{0}
In these section, we will prove a generalized version of Theorem  \ref{thmRanders2}. Indeed  we study  complete Randers metric $F=\alpha+\beta$ with assumption   $\mathfrak{R}_{0j}=\mathfrak{R}_{i0}=0$ where $\mathfrak{R}_{0j}=\mathfrak{R}_{ij}y^i$ and $\mathfrak{R}_{i0}=\mathfrak{R}_{ij}y^j$. More precisely, we  prove the following.
\begin{thm}\label{thmRanders}
Let $F=\alpha +\beta$ be a complete Randers metric on a  manifold $M$. Suppose that  $\mathfrak{R}_{0i}=\mathfrak{R}_{i0}=0$. Then $F$ is a Berwald metric.
\end{thm}

To prove Theorem \ref{thmRanders},  we are going to establish a relation between the mean Landsberg curvature  of a Randers metric  and  its Ricci tensor  in the case that  $\mathfrak{R}_{0j}=\mathfrak{R}_{i0}=0$.  In this case, we find that the mean Landsberg tensor ${\bf J}$ of Rander  metric satisfy a special equation along geodesics:
\begin{lem}\label{PropRanders}
Let $F=\alpha+\beta$  be a Randers metric on a n-dimensional manifold $M$.  Suppose that the Ricci tensor of $F$ satisfy    $\mathfrak{R}_{0i}=\mathfrak{R}_{i0}=0$. Then for any linearly parallel vector fields $u=u(t)$ and $v=v(t)$ along a geodesic $c(t)$, we have
 \begin{equation}\label{RMS2}
\frac{d}{dt}[{\bf J}_{\dot{c}}(u,v)]=0.
\end{equation}
\end{lem}
\begin{proof} We are specially concerned with the Cartan connection and the $h$- and $v$- covariant derivatives  are denoted by ``$|$" and ``$,$" respectively.  The following Bianchi identity of Cartan connection is hold
 \begin{equation}\label{RMS3}
R^{\ h}_{l \ ij,k}+Q^{\ h}_{l \ kr} R^r_{\ ij}+S_{(ij)}\{R^{\ h}_{l \ ir}C^r_{\ jk}+P^{\ h}_{l \ ir}L^r_{\ jk}+P^{\ h}_{l \ jk|i}\}=0,
\end{equation}
where $S_{(ij)}$ means interchange of indices $i$ and $j$ (for more details  see the formula (17.15) page 113 in \cite{Ma}). We take a  contraction for $h$ and $j$ in  (\ref{RMS3}) and get the following
 \begin{eqnarray}\label{RMS4}
\nonumber R_{lk,i}=P^{\ s}_{l \ sr}L^r_{\ ki}-R_{lr}C^r_{\ ki}\!\!\!\!&-&\!\!\!\!\ P^{\ s}_{l\ si|k}+P^{\ r}_{l\ ki}-R^{\ m}_{l \ kr}C^r_{\ mi}\\ \!\!\!\!&-&\!\!\!\!\ P^{\ m}_{l\ kr}L^r_{\ si}+Q^{\ m}_{l\ ir}R^r_{\ mk},
\end{eqnarray}
(see (2.5) in \cite{Shim}). Contracting (\ref{RMS4}) by $y^l$ yields
 \begin{equation}\label{RMS5}
R_{0k,i}=R_{ik}-R_{0r}C^r_{\ ki}-J_{i|k}+L^r_{\ ki}J_r+L^r_{\ ki|r}-R^s_{\ kr}C^r_{\ si}-L^s_{\ kr}L^r_{\ si}.
\end{equation}
For more details see $(2.5)'$ in \cite{Shim}. Since $F$ is a Randers metric, then  $F$ is C-reducible
\begin{equation}\label{RMS6}
C_{ijk}={1\over n+1} \{ I_i h_{jk} + I_j h_{ik} + I_k h_{ij}\}.
\end{equation}
Taking a horizontal derivative of  (\ref{RMS6}), using $h_{ij|k}=0$ and $C_{ijk|s}y^s=L_{ijk}$ we get
\begin{equation}\label{RMS7}
L_{ijk}={1\over n+1}\{ J_i h_{jk} + J_j h_{ik} +J_k h_{ij}\}.
\end{equation}
From (\ref{RMS7}) we derive
\begin{equation}\label{RMS8}
L^r_{\ ki|r}={1\over n+1}\{J^r_{\ |r}h_{ki}+J_{i|k}+J_{k|i}-F^{-1}(J_{i|s}y^s\ell^k+J_{k|s}y^s\ell_i)\}.
\end{equation}
Putting  (\ref{RMS6}), (\ref{RMS7}) and (\ref{RMS8}) in (\ref{RMS5}), after long computations, we have
\begin{eqnarray}\label{RMS9}
\nonumber R_{0k,i}\!\!\!\!&=&\!\!\!\!\ {S_{(ik)}\over n+1} \{J_{i|k}-F^{-1}J_i\ell_k+\frac{(n-3)}{2}J_kJ_i-\frac{1}{n+1}(R_{0k}-F^{-1}R_{00}\ell_k)I_i\}\\
\!\!\!\!&+&\!\!\!\!\ {1\over n+1}(J^r_{\ |r}+(n-1)J^rJ_r-{1\over n+1}I^mR_{0m})h_{ik}+\mathfrak{R}_{ik},
\end{eqnarray}
where
\begin{equation}\label{RMS1}
\mathfrak{R}_{ik}:=R_{ik}-\frac{1}{n+1}\Big(I^mR_{ikm}+I^mR_{mki}+F^{-1}\ell_iI^mR_{m0k}+I_iR_{0k}\Big).
\end{equation}
Multiplying  (\ref{RMS9}) with  $y^k$ yields
\begin{equation}\label{RMS10}
R_{00,i}=R_{i0}+R_{0i}-\frac{1}{n+1}(2I^mR_{m0i}+R_{00}I_i)-J_{i|s}y^s.
\end{equation}
Contracting (\ref{RMS1}) by $y^i$ and $y^k$, we get respectively
 \begin{eqnarray}
\mathfrak{R}_{0i}\!\!\!\!&=&\!\!\!\!\ R_{0i},\label{RMS11}\\
\mathfrak{R}_{i0}\!\!\!\!&=&\!\!\!\!\ R_{i0}-{1\over n+1}(2I^mR_{m0i}+I_iR_{00}).\label{RMS11*}
\end{eqnarray}
By assumption we have $\mathfrak{R}_{i0}=\mathfrak{R}_{0i}=0$. From (\ref{RMS11}) and (\ref{RMS11*}) it follows that  \begin{eqnarray}
R_{0i}\!\!\!\!&=&\!\!\!\!\ 0,\label{RMS12}\\
R_{i0}\!\!\!\!&=&\!\!\!\!\ {2\over n+1}I^mR_{m0i}.\label{RMS12*}
\end{eqnarray}
By considering (\ref{RMS12}) and (\ref{RMS12*}),  the equation  (\ref{RMS10}) reduces to  (\ref{RMS2}).
\end{proof}

\begin{lem}\label{equation}
Let $(M,F)$ be Finsler manifold. Suppose that the mean Landsberg curvature of $F$ satisfy  equation  (\ref{RMS2}). Then for any geodesic c(t) and any parallel vector field V(t) along $c$, the following function
\begin{equation}\label{RMS13}
{\bf I}(t):= {\bf I}_{\dot c}(V(t))
\end{equation}
must be in the following forms
\begin{equation}\label{RMS14}
{\bf I}(t)=t\  {\bf J}(0)+ {\bf I}(0).
 \end{equation}
\end{lem}
\begin{proof} By assumption we have
\begin{equation}\label{RMS15}
J_{i|l}y^l=I_{i|l|s}y^ly^s=0.
\end{equation}
Let
\begin{equation}\label{RMS16}
{\bf J}(t):={\bf J}_{\dot c}(V(t))
\end{equation}
From our definition of ${\bf J}_y$, we get
\begin{equation}\label{RMS17}
{\bf J}(t)={\bf I}^{'}(t).
\end{equation}
By (\ref{RMS15}) we conclude
\begin{equation}\label{RMS18}
{\bf I}^{''}(t)=0,
\end{equation}
which implies that ${\bf I}^{'}(t)={\bf I}^{'}(0)$. By (\ref{RMS17}), we get (\ref{RMS14}).
\end{proof}

\begin{rem}\label{Remark}
\emph{Let $(M,F)$ be a Finsler space and $c: [a,b]\rightarrow M$ be a
geodesic. For a parallel  vector field $V(t)$ along the geodesic $c$,}
\begin{equation}\label{RMS19}
g_{\dot c}(V(t),V(t))= constant.
\end{equation}
\end{rem}

\bigskip

\noindent {\bf Proof of Theorem \ref{thmRanders}}: To prove Theorem \ref{thmRanders}, take an arbitrary unit vector $y\in T_xM$ and an arbitrary vector $v\in T_xM$. Let $c(t)$ be the geodesic with $\dot c(0)=y$ and $V(t)$ the parallel vector field along $c$ with $V(0)=v$. Define ${\bf I}(t)$ and ${\bf J}(t)$ as in (\ref{RMS13}) and (\ref{RMS16}),
respectively. Then by Lemma \ref{equation}, we have
\begin{equation}\label{RMS20}
{\bf I}(t)=t{\bf J}(0)+{\bf I}(0).
\end{equation}
The norm of mean Cartan torsion at a point $x\in M$ is defined by
 \begin{equation}\label{RMS21}
\|{\bf I}\|_x:=\sup_{0\neq y\in T_xM}\sqrt{I_i(x,y)g^{ij}(x,y)I_i(x,y)}.
\end{equation}
It is known that if $F=\alpha+\beta$ is a Randers metric, then
 \begin{equation}\label{RMS22}
\|{\bf I}\|_x\leq\frac{n+1}{2}\sqrt{1-\sqrt{1-\|\beta\|^2_x}}<\frac{n+1}{\sqrt{2}}.
\end{equation}
See \cite{ShKS} or \cite{Sh1} for a proof. So ${\bf I}_y$ is bounded, i.e., there is a constant $N<\infty$ such that
\begin{equation}\label{RMS23}
||{\bf I}||_x:=\sup_{y\in T_xM_0}\sup_{v\in T_xM}{\frac{{\bf I}_y(v)}{[g_y(v,v)]^{\frac{3}{2}}}\leq N}
\end{equation}
By Remark \ref{Remark},  $Q:=g_{\dot c}(V(t),V(t))= constant$ is positive constant. Thus
\[
|{\bf I}(t)|\leq NQ^{\frac{3}{2}}<\infty,
\]
and ${\bf I}(t)$ is a bounded function on $[0,\infty)$. Using $||{\bf I}||<\infty$ and letting $t\rightarrow +\infty$ or $t\rightarrow -\infty$,  we conclude that
\begin{equation}\label{RMS24}
{\bf J}_y(v)={\bf J}(0)=0.
\end{equation}
Therefore ${\bf J}=0$ and $F$ is a weakly Landsberg metric. From (\ref{RMS7}), we conclude that $F$ is a Landsberg metric.  It is known that $F=\alpha+\beta$ is a Landsberg metric if and only if  is a Berwald metric \cite{M3}. Then $F$ is a Berwald metric.
\qed

\smallskip

\begin{cor}
Let $(M, F)$ be a compact negatively curved Finsler manifold of dimension $n\geq 3$. Suppose that $F$ is of  non-zero scalar flag curvature ${\bf K}$. If\ \  $\mathfrak{R}_{0i}=\mathfrak{R}_{i0}=0$, then $F$ is a Riemannian metric.
\end{cor}
\begin{proof}
In  \cite{MS}, Mo-Shen prove that every closed n-dimensional Finsler manifold of negative  scalar flag curvature  is  a Randers metric $(n\geq 3)$. By Theorem \ref{thmRanders}, $F$ is a Berwald metric. Numata showed that every Landsberg metric of non-zero scalar flag curvature is Riemannian \cite{Nu}. This completes the proof.
\end{proof}

\begin{cor}
Let $(M, F)$  be a complete Randers manifold  with vanishing $\mathfrak{R}$-tensor. Suppose that $F$ is of  non-zero scalar flag curvature. Then $F$ is a Riemannian metric.
\end{cor}
\begin{proof}
By Theorem \ref{thmRanders2}, $F$ is a Berwald metric. According to Akbar-Zadeh's theorem,  every complete Berwald metric of non-zero scalar flag curvature with bounded Cartan tensor is Riemannian   \cite{AZ}. This completes the proof.
\end{proof}

\section{Proof of Theorem \ref{RS2}.}\setcounter{equation}{0}

\noindent{\bf  Proof of Theorem \ref{RS2}}:
Let $F=\alpha+\beta$ is of scalar flag curvature ${\bf K}$. In this case, it is known that the Ricci tensor $R_{ij}$ is symmetric and we have
\begin{equation}\label{1-1}
R_{ijk}=\frac{F^2}{3}({\bf K}_{,j} h_{ik}-{\bf K}_{,k} H_{ij})+{\bf K}(y_jh_{ik}-y_kh_{ij}),
\end{equation}
where ${\bf K}_{,i}=\frac{\partial {\bf K}}{\partial y^i}$ \cite{AZ}. Suppose that  $F$ be a non-Riemannian Randers metric of scalar flag curvature ${\bf K}$. Then $\mathfrak{R}_{ij}$ is written in the following form
\begin{equation}\label{1-2}
\mathfrak{R}_{ij}=R_{ij}+\frac{F^2}{3(n+1)}({\bf K}h_{ij}+{\bf K}_{,i}I_j+{\bf K}_{,j}I_i)-\frac{I_i}{3}(F^2{\bf K}_{,j}+3{\bf K}y_j).
\end{equation}
Let $\mathfrak{R}_{ij}$ is symmetric. Recall that  ${\bf K}_{,i}y^i=0$.  Since $I_i\neq 0$, then we have ${\bf K}=0$. Conversely, it is easy to see that if  ${\bf K}=0$ then $\mathfrak{R}_{ij}$ is symmetric.
\qed
\bigskip

\begin{cor}
Let $(M,F)$ be a non-Riemannian Randers manifold. Suppose that $F$ is of scalar flag curvature ${\bf K}$. Then ${\bf K}=0$ if and only if $\mathfrak{R}_{i0}=\mathfrak{R}_{0i}$.
\end{cor}
\begin{proof}
By (\ref{1-2}) we have
 \begin{equation}
\mathfrak{R}_{i0}=R_{i0}-{\bf K}F^2I_i,\ \ \ \mathfrak{R}_{0i}=R_{0i}.\label{SFC1}
\end{equation}
For a Finsler metrics of scalar flag curvature,  the Ricci tensor is symmetric. Then  by (\ref{SFC1}) we have ${\bf K}F^2I_i=0$. Since $I_i\neq 0$,  then ${\bf K}=0$.
\end{proof}

Akbar Tayebi\\
Faculty  of Science, Department of Mathematics\\
Qom University\\
Qom, Iran.\\
Email:\ akbar.tayebi@gmail.com
\bigskip

\noindent
Esmaeil Peyghan\\
Faculty  of Science, Department of Mathematics\\
Arak University\\
Arak. Iran\\
Email: epeyghan@gmail.com
\end{document}